\documentclass[a4paper,11pt, reqno]{amsart}
\usepackage{amssymb,amsthm,amsmath}

\usepackage{enumerate}
\usepackage{ifthen}
\usepackage{graphicx} 
\usepackage[colorlinks,citecolor=blue,hypertexnames=false]{hyperref}

 \numberwithin{equation}{section}
\newtheorem{thm}{Theorem}[section]
\newtheorem{cor}[thm]{Corollary}
\newtheorem{lem}[thm]{Lemma}

\newtheorem{defn}[thm]{Definition}

\allowdisplaybreaks[2]

\theoremstyle{definition}
\newtheorem{rmk}[thm]{Remark}

{\qed\bigskip}

\newcounter{alphabet}

\ifx\undefined\bysame
\newcommand{\bysame}{\leavevmode\hbox to3em{\hrulefill}\,}
\fi

\title[Fractional wave equation on compact Lie groups]
{Nonlinear fractional wave equation on compact Lie groups}
\author{Aparajita Dasgupta} 
\address{Aparajita Dasgupta, Assistant Professor \endgraf Department of Mathematics
	\endgraf Indian Institute of Technology  Delhi
	\endgraf Delhi, 110016  India.} 
\email{adasgupta@maths.iitd.ac.in}
\author{Vishvesh Kumar} 

\address{Vishvesh Kumar, Ph. D.  \endgraf Department of Mathematics: Analysis, Logic and Discrete Mathematics
	\endgraf Ghent University
	\endgraf Krijgslaan 281, Building S8,	B 9000 Ghent,
	Belgium.} 
\email{vishveshmishra@gmail.com}

\author{Shyam Swarup Mondal} 

\address{Shyam Swarup Mondal  \endgraf Department of Mathematics
	\endgraf Indian Institute of Technology Delhi
	\endgraf Delhi, 110 016, India.} 
\email{mondalshyam055@gmail.com}

\thanks{ The first and third authors were supported by Core Research Grant(RP03890G), Science and Engineering Research Board (SERB), DST, India.  The second author was supported  by the FWO Odysseus 1 grant G.0H94.18N: Analysis and Partial Differential Equations, the Methusalem programme of the Ghent University Special Research Fund (BOF) (Grant number 01M01021), and by FWO Senior Research Grant G011522N.}
\date{\today}
\begin{document}
	\allowdisplaybreaks

	\begin{abstract} 
		
		Let $G$ be a compact Lie group. In this article, we consider the initial value fractional wave equation with power-type nonlinearity on    $G$. Mainly, we investigate some  $L^{2}-L^{2}$ estimates of the solutions to the homogeneous fractional wave equation on $G$  with the help of the group Fourier transform on $G$.  	Further,   using the Fourier analysis on compact Lie groups, we prove a local in-time existence result in the energy space. Moreover,  under certain conditions on the initial data, a finite time blow-up result is established. We also derive a    sharp lifespan for local (in-time) solutions. 	Finally, we consider the space-fractional wave equation with a regular mass term depending on the position and study the well-posedness of the fractional Klein-Gordon equation on compact Lie groups.	 	\end{abstract}

	\maketitle
	\section{Introduction}
	In this paper, we investigate a  finite time blow-up result for solutions to fractional wave equations involving the Laplace-Beltrami operator on the compact Lie group under a suitable sign assumption for the initial data.
	
	For $0<\alpha<1$,  we consider the Cauchy problem for fractional wave equation  with power type nonlinearity, namely,
	\begin{align} \label{eq0010}
		\begin{cases}
			\partial^2_tu+(-\mathcal{L})^\alpha u =f(u), & x\in G,t>0,\\
			u(0,x)=\varepsilon u_0(x),  & x\in G,\\ \partial_tu(x,0)=\varepsilon u_1(x), & x\in G,
		\end{cases}
	\end{align}
	where 	  $G$ is a compact Lie group,   $\mathcal{L}$ be the Laplace-Beltrami operator on $G$ (which also coincides with the Casimir element of the enveloping algebra), and $\varepsilon$ is a positive constant describing the smallness of the Cauchy data. Here  for the moment, we assume that  $u_{0}$ and $ u_{1}$ are  taken from the energy space $ H_{\mathcal{L}}^\alpha(G)$ and $ L^2(G)$,  respectively, and  concerning the nonlinearity $f(u)$, we  deal only with  the typical case such as $f(u):=|u|^{p}, p>1$ without losing the  essence of the problem.

	The study of partial differential equations is undoubtedly one of the fundamental tools for understanding and modeling natural and real-world phenomena.  Fractional differential operators are nonlocal operators that are considered as a generalization of classical differential operators of arbitrary non-integer orders.
	For the last few decades,   partial differential equations involving nonlocal operators have gained a considerable amount of interest and have become one of the essential topics in mathematics and its applications.  Many physical phenomena in engineering,    quantum field theory, astrophysics,  biology, materials, control theory, and other sciences can be successfully described by models utilizing mathematical tools from fractional calculus \cite{Neww1, NN1,NN3,NN4, f4}.    In particular, the fractional Laplacian is represented as the infinitesimal generator of stable radially symmetric Lévy processes  \cite{EE1}.  
	For other exciting models related to fractional differential equations, we refer to the reader  \cite{f3, f10,f11,EE24} to mention only a few of many recent publications.


	In recent years, due to the nonlocal nature of the fractional derivatives, considerable attention has been devoted to various models involving fractional Laplacian and nonlocal operators by several researchers.  	There is a vast literature available involving the fractional Laplacian on the Euclidean framework, which is difficult to mention; we refer to important papers     \cite{EE3, EE4, EE5, EE6, EE10, EE24, EE27, EE31} and the references therein.  Here we would like to point out that the fractional Laplacian operator $(-\Delta)^{\alpha}$ can be reduced to the classical Laplace operator $-\Delta$ as $\alpha \rightarrow 1$. We refer to \cite{EE24} for more details. In particular,  many interesting results in some classical elliptic problems have been extended in the fractional Laplacian setting, see 	\cite{EE9}. We also refer \cite{mi}  and \cite{mii} for the Fractional Klein-Gordon equation with singular mass on the Euclidean and the graded Lie group, respectively.

	The study of the semilinear wave equation has also been extended in the non-Euclidean framework. Several papers are devoted for studying linear PDE in non-Euclidean structures in the last decades.   For example, the semilinear wave equation with or without damping has been investigated for the   Heisenberg group \cite{24,30}.    In the case of graded groups, we refer to the recent works \cite{gra1, gra2, gra3}.  	Concerning the damped wave equation on compact Lie groups, we refer to \cite{27, 28,31,garetto}. Here, we would also like to highlight that the discrete time-dependent wave equation and it's semiclassical analysis were considered by   Dasgupta,  Ruzhansky, and Tushir \cite{apara}. They also studied   the Klein-Gordon equation with discrete  fractional Laplacian on $h\mathbb{Z}^n$  in \cite{aparajita}.

	
	In particular,  the author in \cite{27} studied semilinear wave equation with power nonlinearity $|u|^p$ on compact Lie groups and proved a local in-time existence result in the energy space via Fourier analysis on compact Lie groups. He also derived a blow-up result for the semilinear Cauchy problem for any $p > 1$. Then, an interesting and viable problem is to study  the fractional wave equation  (\ref{eq0010})  of order $\alpha$ with $ 0 < \alpha <  1$, with power-type nonlinearity on  $G$.   So far, to the best of   our knowledge, the fractional wave equation has not been consider yet in the frame of compact Lie groups. The main aim of this article is to investigate the fractional wave equation with power-type nonlinearity on the compact Lie group $G$. More preciously, using the Gagliardo-Nirenberg type inequality (in order to handle power nonlinearity in $L^2(G))$, we prove the local well-posedness of the   Cauchy problem (\ref{eq0010}) in the energy evolution space  $\mathcal{C}\left([0, T], H_{\mathcal{L}}^{\alpha}( {G})\right) \cap \mathcal{C}^{1}\left([0, T], L^{2}( {G})\right)$. Further,   we establish a finite time blow-up result to  (\ref{eq0010})  for any  $p > 1$   provided the initial data satisfies certain sign assumptions. 	 
	Finally,   we consider a space-fractional wave equation with a regular mass term depending on the position and study the well-posedness of the space-fractional Klein-Gordon equation on the compact Lie group $G$. 	
	%
	%
	%
	%
	%
	%
	%
	%
	%
	%


	\subsection{Main results}
	Throughout the paper  we denote $L^{q}(G), 1 \leq  q<\infty$, the space of $q$-integrable functions on the compact Lie group $G$ with respect to the normalized Haar measure on $G$   and  essentially bounded for $q=\infty$.   For $\alpha>0$ and $q \in(1, \infty)$, the fractional Sobolev space $H_{\mathcal{L} }^{ \alpha, q}(G)$ of order $\alpha$ is defined as  
	\begin{align}\label{sob}
		H_{\mathcal{L}}^{\alpha, q}(G) \doteq\left\{f \in L^{q}(G):(-\mathcal{L})^{\alpha / 2} f \in L^{q}(G)\right\}
	\end{align}
	endowed with the norm $\|f\|_{H_{\mathcal{L}}^{\alpha, q}(G)} \doteq\|f\|_{L^{q}(G)}+\left\|(-\mathcal{L})^{\alpha / 2} f\right\|_{L^{q}(G)}$.  We simply denote   the Hilbert space $H_{\mathcal{L}}^{\alpha, 2}(G)$ by $H_{\mathcal{L}}^{\alpha}(G)$.   
	
	By employing Fourier analysis for compact Lie groups, our first result concerning   $L^2$-decay estimates for the solution of the homogeneous Cauchy problem (\ref{eq0010}) (when $f=0$)   is   stated in the following theorem.
	\begin{thm}\label{thm11}
		Let $0<\alpha <1$.  Suppose that $(u_0, u_1)\in H_{\mathcal{L}}^\alpha(G) \times  L^2(G)$ and $u\in\mathcal{C}([0,\infty),H_{\mathcal{L}}^\alpha(G))\cap \mathcal{C}^1([0,\infty),L^2(G))$ be the solution to the homogeneous Cauchy problem
		\begin{align}\label{eq1}
			\begin{cases}
				\partial^2_tu+(-\mathcal{L})^\alpha u =0, & x\in G,~t>0,\\
				u(0,x)=u_0(x),  & x\in G,\\ \partial_tu(x,0)=u_1(x), & x\in G.
			\end{cases}
		\end{align}
		Then, $u$ satisfies the following $L^2( G)-L^2( G)$ estimates
		\begin{align}\label{111111}
			\| u(t,\cdot)\|_{L^2( G)} &\leq C[ \left\|u_{0}\right\|_{L^{2}(G)}+t\left\|u_{1}\right\|_{L^{2}(G)}],\\\nonumber
			\left\|(-\mathcal{L})^{\alpha / 2} u(t, \cdot)\right\|_{L^{2}(G)}&\leq C[ \left\|u_{0}\right\|_{H_{\mathcal{L}}^{{\alpha }}(G)}+\left\|u_{1}\right\|_{L^{2}(G)}],\\\nonumber
			\|\partial_tu(t,\cdot)\|_{L^2( G)}&\leq C[ \left\|u_{0}\right\|_{H_{\mathcal{L}}^{{\alpha }}(G)}+\left\|u_{1}\right\|_{L^{2}(G)}].
		\end{align}
		for any $t\geq 0$, where  $C$ is a positive multiplicative constant.
	\end{thm}
	%
	%
	

	Next we prove the local well-posedness of the   Cauchy problem   (\ref{eq0010})  in the energy evolution space  $\mathcal C\left([0,T],  H^\alpha_{\mathcal{L}}(G)\right)\cap\mathcal C^1\left([0,T],L^2(G)\right)$.  In this case, a Gagliardo-Nirenberg type inequality (proved in \cite{Gall}) will be used in order to estimate the power nonlinearity in $L^2(G)$.	The following result  is about the   local existence   for the solution of  the Cauchy problem (\ref{eq0010}). 
	\begin{thm}\label{thm22}
		Let $0<\alpha <1$ and let  $G$ be a compact connected Lie group with the topological dimension $n.$ Assume that $n\geq 2[\alpha]+2$. Suppose that $(u_0, u_1)\in H^\alpha_{\mathcal L}(G) \times  L^2(G)$      and $p>1$ such that $p\leq\frac{n}{n-2\alpha}.$ Then there exists $T=T(\varepsilon)>0$ such that the Cauchy problem (\ref{eq0010})   admits a uniquely determined mild solution $$u\in \mathcal{C}([0,T],H^\alpha_{\mathcal L}(G))\cap \mathcal{C}^1([0,T],L^2(G)).$$
		Moreover, the lifespan $T$ satisfies    the following   lower bound estimates
		\begin{align}\label{eq1111}
			T(\varepsilon) \geq 
			C \varepsilon^{-\frac{p-1}{b(p)}} ,
		\end{align}  
		where     the constant $C>0$ is independent of $\varepsilon$ and 
		\begin{align}\label{bp}
			b(p) =\left\{\begin{array}{ll}
				p+1 & \text { if } u_{1} \neq 0 ,\\
				2 & \text { if } u_{1}=0.
			\end{array}\right.
		\end{align}
	\end{thm}

	
	Our next result is about the non-existence of global in-time solutions to (\ref{eq0010}) for any $p > 1$  regardless of the size of initial data. We first introduce a suitable notion of energy solutions for the   Cauchy problem (\ref{eq0010}) before stating the blow-up result.
	\begin{defn}\label{eq332}
		Let $0<\alpha <1$ and $\left(u_{0}, u_{1}\right) \in H_{\mathcal{L}}^{\alpha}(G) \times L^{2}(G)$. For  any $T>0,$ we say that
		$$
		u \in \mathcal{C}\left([0, T), H_{\mathcal{L}}^{\alpha}(G)\right) \cap \mathcal{C}^{1}\left([0, T), L^{2}(G)\right) \cap L_{\text {loc }}^{p}([0, T) \times G)
		$$
		is an energy solution on $[0, T)$ to (\ref{eq0010}) if $u$ satisfies  the following  integral relation:
		\begin{align}\label{eq011}\nonumber
			&\int_{ {G}} \partial_{t} u(t, x) \phi(t, x)  {d} x-\int_{ {G}} u(t, x) (\partial_{s} \phi)(t, x) {d} x   +\varepsilon \int_{G} u_{0}(x)(\partial_{s}\phi)(0, x) \;d x\\\nonumber
			& -\varepsilon \int_{G} u_{1}(x) \phi(0, x) \;d x +\int_{0}^{t} \int_{G} u(s, x)\left(\partial^2_s\phi(s, x) +(-\mathcal{L})^\alpha \phi(s, x) \right) \;d x {~d} s \\&
			=\int_{0}^{t} \int_{G}|u(s, x)|^{p} \phi(s, x) \;d x  {~d} s
		\end{align}
		for any $\phi \in \mathcal{C}_{0}^{\infty}([0, T) \times G)$ and   any $t \in(0, T)$.
	\end{defn}  
	\begin{thm}\label{f6}
		Let  $0<\alpha <1$, $p>1$,  and let $\left(u_{0}, u_{1}\right) \in H_{\mathcal{L}}^{\alpha}(G) \times L^{2}(G)$ be nonnegative and nontrivial functions. Suppose 
		$$u \in \mathcal{C}\left([0, T), H_{\mathcal{L}}^{\alpha}(G)\right) \cap \mathcal{C}^{1}\left([0, T), L^{2}(G)\right) \cap L_{\mathrm{loc}}^{p}([0, T) \times G)$$ be an energy solution to the Cauchy problem   (\ref{eq0010})    with lifespan $T=T(\varepsilon)$. Then there exists a   constant $\varepsilon_{0}=\varepsilon_{0}\left(u_{0}, u_{1}, p\right)>0$ such that for any $\varepsilon \in\left(0, \varepsilon_{0}\right],$ the energy solution $u$ blows up in finite time. Furthermore, the lifespan $T$ satisfies the following   estimates
		\begin{align}\label{eq1112}
			T(\varepsilon) \leq 
			C \varepsilon^{-\frac{p-1}{b(p)}}  ,
		\end{align}
		where the constant $C>0$ is independent of $\varepsilon$ and \begin{align}
			b(p) =\left\{\begin{array}{ll}
				p+1 & \text { if } u_{1} \neq 0 ,\\
				2 & \text { if } u_{1}=0.
			\end{array}\right.
		\end{align}
	\end{thm}
	\begin{rmk}\label{remark}
		From  (\ref{eq1111}) and (\ref{eq1112}),   the sharp lifespan estimates  	for local in time solutions to (\ref{eq0010}) is given by 
		$$
		\begin{array}{ll}
			C \varepsilon^{-\frac{p-1}{p+1}} \leq T(\varepsilon) \leq C \varepsilon^{-\frac{p-1}{p+1}} & \text { if } u_{1} \neq 0 \\
			C \varepsilon^{-\frac{p-1}{2}} \leq T(\varepsilon) \leq C \varepsilon^{-\frac{p-1}{2}} & \text { if } u_{1}=0.
		\end{array}
		$$
		Thus the nontriviality of $u_{1}$ plays a crucial role in the lifespan estimates.
	\end{rmk}
	
	\begin{rmk}
		Here we note that the fractional Laplace-Beltrami operator  $(-\mathcal{L})^{\alpha}$ can be reduced to the classical Laplace-Beltrami operator    $-\mathcal{L}$ as $\alpha \rightarrow 1$ and all our results coincides with the results proved for the   Cauchy problem in \cite{31}.
	\end{rmk}

	Finally, we consider   the fractional  Klein-Gordon equation on  compact Lie group $G$  with  regular  mass term depending on the spatial variable;  namely  for $T>0$ and for $0<\alpha<1$, we consider the following Cauchy problem:
	\begin{align} \label{eq1111111}
		\begin{cases}
			\partial^2_tu(t, x)+\left( -\mathcal{L}\right)^\alpha u(t, x)+m(x)u(t, x)=0, & (t, x)\in[0, T]\times G,\\
			u(0,x)=u_0(x),  & x\in G,\\ \partial_tu(x,0)=u_1(x), & x\in G,
		\end{cases}
	\end{align}
	where $u_{0}(x)$ and $u_{1}(x)$ are two given functions on $G$.	Here  the mass function $m$ is supposed to be  non-negative and regular function, and in this case, we have the following result.
	\begin{thm}\label{thm11111}
		Let $m\in L^\infty (G)$ be a non-negative function.	Suppose that $ u_0\in H_{\mathcal{L}}^\alpha$ and $u_1\in L^2(G)$. Then, there exists a unique solution  $u\in\mathcal{C}([0,T],H_{\mathcal{L}}^\alpha)\cap \mathcal{C}^1([0,T],L^2(G))$ corresponding   to the  homogeneous Cauchy problem (\ref{eq1111111}) and it satisfies the following estimate
		\begin{align}\label{eq2}
			\|u(t, \cdot)\|_{H_{\mathcal{L}}^\alpha(G)}^2+\left\|\partial_{t} u(t, \cdot)\right\|_{L^{2}(G)}^2 \lesssim (1+\|m\|_{L^\infty(G)}) \left[ \|u_0\|_{H_{\mathcal{L}}^\alpha(G)}^2+\left\| u_1\right\|_{L^{2}(G)}^2\right] ,
		\end{align}
		uniformly in $t \in  [0, T]$.
	\end{thm}

	Before studying the nonhomogeneous Cauchy problem (\ref{eq0010}), we first deal with the corresponding homogeneous problem, i.e., when $f=0$. Particularly,  using the group Fourier transform with respect to the spatial variable,  we determine  $L^{2}-L^{2} $ estimates for the solution of the homogeneous fractional wave equation on the compact Lie group $G$.  
	Once we have these estimates,  applying a Gagliardo-Nirenberg type inequality on compact Lie groups \cite{27, 28, 31} (see  also \cite{Gall}  for  Gagliardo-Nirenberg type inequality on a more general frame of connected Lie groups),   we prove  the local well-posedness result for  (\ref{eq0010}). We establish a blow-up result to  (\ref{eq0010}) by using a comparison argument for ordinary differential inequality of second order.


	Apart from the introduction, this paper is organized as follows. In Section \ref{sec2},  we recall   the Fourier analysis on compact Lie groups which will be used frequently throughout the paper for our approach. In Section \ref{sec3},  
	we show an appropriate decomposition of the propagators for the nonlinear equation in the Fourier space. We also prove Theorem \ref{thm11} by deriving some  $L^{2}-L^{2}$ estimates for the solution of the homogeneous fractional wave equation on the compact Lie group $G$.  In Section \ref{sec4}, first we   briefly recall   the notion of mild solutions  in our framework and  prove   the local well-posedness of the   Cauchy problem   (\ref{eq0010})  in the energy evolution space  $\mathcal C\left([0,T],  H^\alpha_{\mathcal{L}}(G)\right)\cap\mathcal C^1\left([0,T],L^2(G)\right)$. In Section \ref{sec5},  under certain conditions on the initial data, a finite time blow-up result is established.    Finally, we consider a space-fractional wave equation with a regular mass term depending on the position and study the well-posedness of the space-fractional Klein-Gordon equation in Section \ref{sec6}.

	\section{Preliminaries} \label{sec2}
	
	\subsection{Notations} 
	throughout the article,  we use the following notations:  the notation $f \lesssim g$ means that there exists a positive constant $C$ such that $f \leq C g$; 		  $C$ denotes a suitable positive constant and may have different value from line to line;  we denote  $G$  as the  compact Lie group;   $\mathcal{L}$ denotes the Laplace-Beltrami operator on $G;$    
	$ \operatorname{Tr}(A)=\sum_{j=1}^{d} a_{j j}$  denotes the trace  of the matrix $A=\left(a_{i j}\right)_{1 \leq i, j \leq  d} $; $I_{d} \in \mathbb{C}^{d \times d}$ denotes the identity matrix;	
	$\;d x$ stands for the normalized Haar measure on the compact group $G$;   and $[x]$ denotes the greatest integer function of $x$.
	\subsection{Group Fourier transform}
	In this subsection, we recall some basics of Fourier analysis on compact Lie groups to make the paper self-contained. A complete account of representation theory and noncommutative Fourier analysis of compact Lie  groups can be found in \cite{garetto, RT13, RuzT}. However, we mainly adopt the notation and terminology given in \cite{RuzT}.

	Let $G$ be a compact Lie group.  	A continuous unitary representation $\xi$ of dimension $d_{\xi}$   is a continuous group homomorphism from $ {G}$ onto the group of unitary matrices,  $ {U}$ of order $d_{\xi}\times d_{\xi}$, i.e., $\xi(x y)=\xi(x) \xi(y)$ and $\xi(x)^{*}=\xi(x)^{-1}$ for all $x, y \in {G}$. 
	Two representations $\xi, \eta$ of ${G}$ are called  equivalent if there exists an invertible intertwining operator $T$ such that $T \xi(x)=\eta(x) T$ for any $x \in {G}$. Also, a subspace $M \subset \mathbb{C}^{d_{\xi}}$ is said to be invariant under the unitary representation $\xi$ if $\xi(x) M \subset M$ for any $x \in {G}$. The unitary representation $\xi$ is said to be irreducible if the only   $\xi$ invariant closed subspaces of $ \mathbb{C}^{d_{\xi}}$ are  $\{0\}$ and $ \mathbb{C}^{d_{\xi}}$.

	The unitary dual of ${G}$, denoted by $\widehat{{G}}$,  is  the collection of  equivalence classes  $[\xi]$  of continuous irreducible unitary representation $\xi: {G} \rightarrow \mathbb{C}^{d_{\xi} \times d_{\xi}}$.  Since $G$ is compact, the set $\widehat{G}$ is discrete. Thus for  $[\xi] \in \widehat{G}$, by choosing a basis in the representation space of $\xi$, one  can view $\xi$ as a matrix-valued function $\xi: G \rightarrow \mathbb{C}^{d_{\xi} \times d_{\xi}}$.
	By the Peter-Weyl theorem, the collection
	$$
	\left\{\sqrt{d_{\xi}} \xi_{i j}: 1 \leq i, j \leq d_{\xi},[\xi] \in \widehat{G}\right\}
	$$
	is an orthonormal basis of $L^{2}(G)$, where the matrix coefficients $ \xi_{i j}:G\to \mathbb{C}$ of $\xi$ are continuous functions for all $i, j \in\left\{1, \ldots, d_{\xi}\right\}$.
	
	Let  $f \in L^{1}(G)$. Then the group Fourier transform of $f$ at $\xi\in \widehat{G}$ is defined by
	$$
	\widehat{f}(\xi):=\int_{G} f(x) \xi(x)^{*} d x,
	$$
	where $d x$ is the normalised Haar measure on $G$.  Since  $\xi$ is a matrix representation, we have $\widehat{f}(\xi) \in \mathbb{C}^{d_{\xi} \times d_{\xi}}$.  If $f\in L^2(G)$, by the Peter-Weyl theorem,   the Fourier series representation for $f$ is given by
	$$
	f(x)=\sum_{[\xi] \in \widehat{G}} d_{\xi} \operatorname{Tr}(\xi(x) \widehat{f}(\xi)).
	$$
	Moreover, for $f \in L^2(G)$,  Plancherel formula on  $G$ takes the following  form
	\begin{align}\label{eq002}
		\|f\|_{L^{2}(G)}=\left(\sum_{[\xi] \in \widehat{G}} d_{\xi}\|\widehat{f}(\xi)\|_{\mathrm{HS}}^{2}\right)^{1 / 2},
	\end{align}
	where the Hilbert-Schmidt norm of  $\widehat{f}(\xi) $ is defined  as 
	$$	\|\widehat{f}(\xi)\|_{\mathrm{HS}}^{2}=\operatorname{Tr}\left(\widehat{f}(\xi) \widehat{f}(\xi)^{*}\right)=\sum_{i, j=1}^{d_{\xi}}|\widehat{f}(\xi)_{ij}|^2 ,$$
	which gives a norm on $\ell^{2}(\widehat{G})$.

	Let $\mathcal{L}$ be the  Laplace-Beltrami operator on the compact Lie group $G$. 	It is   important to understand the behavior of the group Fourier transform with respect to the Laplace–Beltrami operator $\mathcal{L}$ for our analysis.  For  $[\xi] \in \widehat{{G}}$,   the matrix elements   $\xi_{i j}$,  are  the eigenfunctions of $\mathcal{L}$ with the same   eigenvalue $-\lambda_{\xi}^{2}$, i.e.,    for any $ x \in {G}$
	$$
	-\mathcal{L} \xi_{i j}(x)=\lambda_{\xi}^{2} \xi_{i j}(x),   \qquad \text{for all } i, j \in\left\{1, \ldots, d_{\xi}\right\}.
	$$
	In other words, the symbol of  the  Laplace-Beltrami operator $\mathcal{L}$ is given by 
	\begin{align}\label{symbol}
		\sigma_{\mathcal{L}}(\xi)=-\lambda_{\xi}^{2} I_{d_{\xi}},
	\end{align}
	for any $[\xi] \in \widehat{{G}}$,
	where $I_{d_{\xi}} \in \mathbb{C}^{d_{\xi} \times d_{\xi}}$ denotes the identity matrix. Thus   $$\widehat{\mathcal{L} f}(\xi)=\sigma_{\mathcal{L}}(\xi) \widehat{f}(\xi)=-\lambda_{\xi}^{2} \widehat{f}(\xi)$$ for any $[\xi] \in \widehat{ G}$. 
	Further,  using the Plancherel formula,  for any $\alpha>0$, we have 
	$$
	\left\|(-\mathcal{L})^{\alpha / 2} f\right\|_{L^{2}({G})}^{2}=\sum_{[\xi] \in \widehat{{G}}} d_{\xi} \lambda_{\xi}^{2 \alpha}\|\widehat{f}(\xi)\|_{\mathrm{HS}}^{2}
	.	 	$$
	For $\alpha>0,$ the   Sobolev space $H_{\mathcal{L}}^\alpha\left(G\right)$ is defined as follows: 
	$$H_{\mathcal{L}}^\alpha(G)=\left\{u \in L^{2}(G):\|u\|_{H_{\mathcal{L}}^\alpha(G)}<+\infty\right\},$$ where $\|u\|_{H_{\mathcal{L}}^\alpha(G)}=\|u\|_{L^{2}(G)}+\left\|(-\mathcal{L})^{\alpha / 2} u\right\|_{L^{2}({G})}$ and    $(-\mathcal{L})^{\alpha / 2} $  is     the fractional Laplace-Beltrami operator  defined in terms of the Fourier transform, i.e., 
	$$(-\mathcal{L})^{\alpha / 2} f =\mathcal{F}^{-1}\left(\lambda_{\xi}^{2 \alpha }(\mathcal{F} u)\right),  	\quad  \text{for all $[\xi] \in \widehat{{G}}$}.$$

	

	\section{Fourier multiplier expressions and $L^2(G)-L^2(G)$ estimates 
	}\label{sec3}
	In this section, we derive $L^2(G)– L^2(G)$ estimates for the solutions of  the homogeneous problem     (\ref{eq1}).
	We  employ the group Fourier transform on the compact group $G$ with respect to  the space variable $x$  together with the Plancherel identity in order to      estimate     $L^2$-norms of  $u(t, ·), (-\mathcal{L})^{\frac{\alpha}{2}}u(t, \cdot)$, and $\partial_{t}u(t, ·)$. 
	
	Let $u$ be a solution to (\ref{eq1}). Let $\widehat{u}(t, \xi)=(\widehat{u}(t, \xi)_{kl})_{1\leq k, l\leq d_\xi}\in \mathbb{C}^{d_\xi\times d_\xi}, [\xi]\in\widehat{ G}$ denote the Fourier transform of $u$  with respect to the $x $ variable. Invoking the group Fourier transform with respect to $x$ on   (\ref{eq1}), we deduce that $\widehat{u}(t, \xi)$ is  a solution to the following  Cauchy problem for the system of ODE's (with size of the system that depends on the representation $\xi$)
	\begin{align}\label{eq6661}
		\begin{cases}
			\partial^2_t\widehat{u}(t,\xi)+(-	\sigma_{\mathcal{L}}(\xi))^\alpha \widehat{u}(t,\xi) =0,& [\xi]\in\widehat{ G},~t>0,\\ \widehat{u}(0,\xi)=\widehat{u}_0(\xi), &[\xi]\in\widehat{ G},\\ \partial_t\widehat{u}(0,\xi)=\widehat{u}_1(\xi), &[\xi]\in\widehat{ G},
		\end{cases} 
	\end{align}
	where  $\sigma_{\mathcal{L}}$	is the symbol of  the  operator operator $\mathcal{L}$.  Using the identity (\ref{symbol}),  the   system  (\ref{eq6661}) can be written in the form of $d_\xi^2$ independent   ODE's, namely,
	\begin{align}\label{eqq7}
		\begin{cases}
			\partial^2_t\widehat{u}(t,\xi)_{kl}+ \lambda_\xi^{2\alpha }  \widehat{u}(t,\xi)_{kl}= 0,& [\xi]\in\widehat{ G},~t>0,\\ \widehat{u}(0,\xi)_{kl}=\widehat{u}_0(\xi)_{kl}, &[\xi]\in\widehat{ G},\\ \partial_t\widehat{u}(0,\xi)_{kl}=\widehat{u}_1(\xi)_{kl}, &[\xi]\in\widehat{ G},
		\end{cases}
	\end{align}
	for all $k,l\in\{1,2,\ldots,d_\xi\}.$ 	Then    the characteristic equation of (\ref{eqq7}) is given by
	\[\lambda^2+\lambda_\xi^{2\alpha } =0,\]
	and consequently the characteristic roots are   $\lambda=\pm i \lambda_\xi^\alpha  $. 
	%
	Thus the solution of the   homogeneous  problem   (\ref{eqq7}) is given by 
	\begin{align}\label{number2}
		\widehat{u}(t,\xi)_{kl}=A_0(t, \xi) \widehat{u}_0(\xi)_{kl}+A_1(t, \xi) \widehat{u}_1(\xi)_{kl},
	\end{align}
	where
	\begin{align}\label{number1}
		&A_0(t, \xi)=\begin{cases}
			\cos \left(t \lambda_\xi^\alpha \right)& \text{if }\lambda_\xi^\alpha\neq 0,\\
			1& \text{if }\lambda_\xi^\alpha= 0,
		\end{cases} & A_1(t, \xi)=\begin{cases}
			\frac{\sin \left(t \lambda_\xi^\alpha \right)}{\lambda_\xi^\alpha }& \text{if }\lambda_\xi^\alpha\neq 0,\\
			t& \text{if }\lambda_\xi^\alpha= 0.
		\end{cases} 
	\end{align}
	We 	notice that $A_0(t, \xi) = \partial_t A_1(t, \xi)$ for any $[\xi] \in \widehat{ G}$. We  also note that $0$ is an eigenvalue for the continuous irreducible unitary representation $1 : x \in G \to  1 \in  \mathbb{C}$.  Now we        estimate     $L^2$-norms of  $u(t, ·), (-\mathcal{L})^{\frac{\alpha}{2}}u(t, \cdot)$, and $\partial_{t}u(t, ·)$.

	\textbf{Estimate for $\|u(t, \cdot )\|_{L^{2}(G)}$:}
	From the equations (\ref{number2}) and (\ref{number1}), it follows that 
	\begin{align*}
		\left|\widehat{u}(t, \xi)_{k \ell}\right| &\leq |A_0(t, \xi)| \left|\widehat{u}_{0}(\xi)_{k \ell}\right|
		+ |A_1(t, \xi)|  \left|\widehat{u}_{1}(\xi)_{k \ell}\right|  \\&\leq    \left|\widehat{u}_{0}(\xi)_{k \ell}\right|+  t\left|\widehat{u}_{1}(\xi)_{k \ell}\right|, \quad \text { for any } t \geq 0.
	\end{align*}
	Consequently,   using Plancherel formula (\ref{eq002}), we obtain
	\begin{align}\label{L2}
		\|u(t, \cdot)\|_{L^{2}(G)}^{2}&=\sum_{[\xi] \in \widehat{G}} d_{\xi} \sum_{k, \ell=1}^{d_{\xi}}\left|\widehat{u}(t, \xi)_{k \ell}\right|^{2} \nonumber \\ \lesssim &\sum_{[\xi] \in \widehat{G}} d_{\xi} \sum_{k, \ell=1}^{d_{\xi}}\left(\left|\widehat{u}_{0}(\xi)_{k \ell}\right|^{2}+t^{2}\left|\widehat{u}_{1}(\xi)_{k \ell}\right|^{2}\right) \nonumber \\
		&=\left\|u_{0}\right\|_{L^{2}(G)}^{2}+t^{2}\left\|u_{1}\right\|_{L^{2}(G)}^{2} .
	\end{align}
	
	\textbf{Estimate for $\left\|(-\mathcal{L})^{\alpha / 2} u(t, \cdot )\right\|_{L^{2}(G)}$:}
	Using  Plancherel formula, we  get
	\begin{align}\label{f1}
		\left\|(-\mathcal{L})^{\alpha / 2} u(t, \cdot)\right\|_{L^{2}(G)}^2 \nonumber &=\sum_{[\xi] \in \widehat{G}} d_{\xi}  \left\|\sigma_{(-\mathcal{L})^{\alpha / 2}}(\xi)\widehat{u}(t, \xi) \right\|_{HS}^{2} \\=&\sum_{[\xi] \in \widehat{G}} d_{\xi} \sum_{k, \ell=1}^{d_{\xi}}\lambda_\xi^{2\alpha }\left|\widehat{u}(t, \xi)_{k \ell}\right|^{2}.\end{align}
	Now from     (\ref{number2}) and (\ref{number1}), it follows that 
	$$
	\lambda_\xi^{\alpha }\left|\widehat{u}(t, \xi)_{k \ell}\right| \leq 	\lambda_\xi^{\alpha } \left|\widehat{u}_{0}(\xi)_{k \ell}\right|+\left|\widehat{u}_{1}(\xi)_{k \ell}\right|. 
	$$
	Thus  by (\ref{f1}) and  the Plancherel identity, we obtain 
	
	\begin{align}\label{-L}\nonumber
		\left\|(-\mathcal{L})^{\alpha / 2} u(t, \cdot)\right\|_{L^{2}(G)}^2&=\sum_{[\xi] \in \widehat{G}} d_{\xi} \sum_{k, \ell=1}^{d_{\xi}}\lambda_\xi^{2\alpha }\left|\widehat{u}(t, \xi)_{k \ell}\right|^{2}\\\nonumber
		& \lesssim \sum_{[\xi] \in \widehat{G}} d_{\xi} \sum_{k, \ell=1}^{d_{\xi}}\left(\lambda_{\xi}^{{2\alpha }}\left|\widehat{u}_{0}(\xi)_{k \ell}\right|^{2}+\left|\widehat{u}_{1}(\xi)_{k \ell}\right|^{2}\right) \\
		&=\left\|u_{0}\right\|_{H_{\mathcal{L}}^{{\alpha }}(G)}^{2}+\left\|u_{1}\right\|_{L^{2}(G)}^{2}.
	\end{align}
	
	\textbf{Estimate for $\left\|\partial_{t} u(t, \cdot )\right\|_{L^{2}(G)}$:} From (\ref{number2}) and  (\ref{number1}),      for any $[\xi] \in \widehat{G}$ and any $k, \ell \in\left\{1, \ldots, d_{\xi}\right\}$, an 	elementary computation  gives that 
	$$
	\partial_{t} \widehat{u}(t, \xi)_{k \ell}=-\lambda_{\xi}^{2\alpha } A_{1}(t, \xi) \widehat{u}_{0}(\xi)_{k\ell}+A_{0}(t, \xi) \widehat{u}_{1}(\xi)_{k \ell}.
	$$
	Thus 
	$$
	\left|\partial_{t} \widehat{u}(t, \xi)_{k \ell}\right| \leq \lambda_{\xi}^{\alpha}\left|\widehat{u}_{0}(\xi)_{k \ell}\right|+\left|\widehat{u}_{1}(\xi)_{k \ell}\right| .
	$$
	Thus the Plancherel formula yields that 
	\begin{align}\label{deri}\nonumber
		\left\|\partial_{t} u(t, \cdot )\right\|_{L^{2}(G)}^2 & \lesssim \sum_{[\xi] \in \widehat{G}} d_{\xi} \sum_{k, \ell=1}^{d_{\xi}}\left(\lambda_{\xi}^{{2\alpha }}\left|\widehat{u}_{0}(\xi)_{k \ell}\right|^{2}+\left|\widehat{u}_{1}(\xi)_{k \ell}\right|^{2}\right) \\
		&=\left\|u_{0}\right\|_{H_{\mathcal{L}}^{{\alpha }}(G)}^{2}+\left\|u_{1}\right\|_{L^{2}(G)}^{2}.
	\end{align}

	Now, we are in a position to prove  Theorem  \ref{thm11}.

	\begin{proof}[Proof of theorem \ref{thm11}]
		The proof of Theorem \ref{thm11} follows from  the       estimates  (\ref{L2}), (\ref{-L}), and  (\ref{deri})  for $\|u(t, \cdot )\|_{L^{2}(G)}$, $\left\|(-\mathcal{L})^{\alpha / 2} u(t, \cdot )\right\|_{L^{2}(G)}$, and 	$\left\|\partial_{t} u(t, \cdot )\right\|_{L^{2}(G)}$, respectively.
	\end{proof}

	\section{Local existence}\label{sec4}
	This section is devoted to  prove  Theorem \ref{thm22}, that is,   the local well-posedness of the   Cauchy problem   (\ref{eq0010})  in the energy evolution space  $\mathcal C\left([0,T],  H^\alpha_{\mathcal{L}}(G)\right)\cap\mathcal C^1\left([0,T],L^2(G)\right)$.  

	%
	To present the proof of Theorem \ref{thm22}, first we recall some notations. Consider the space \[X(T):=\mathcal{C}\left([0,T],  H^\alpha_{\mathcal L}(G)\right)\cap\mathcal C^1\left([0,T],L^2(G)\right),\] equipped with the norm
	\begin{align}\label{eq33333} 
		\|u\|_{X(T)}&:=\sup\limits_{t\in[0,T]}\left ( a(t)^{-1}\|u(t,\cdot)\|_{L^2(G)}+\|(-\mathcal L)^{\alpha/2}u(t,\cdot)\|_{L^2(G)}+\|\partial_tu(t,\cdot)\|_{L^2(G)}\right ),
	\end{align}
	where
	$$
	a(t) =\left\{\begin{array}{ll}
		1+t & \text { if } u_{1} \neq 0, \\
		1 & \text { if } u_{1}=0.
	\end{array}\right.
	$$
	Here we would like to note that the factor $ a(t)^{-1}$ inside the norm  in  (\ref{eq33333}) appears due to the estimate of  $	\|u(t, \cdot)\|_{L^{2}(G)}\lesssim a(t)( \left\|u_{0}\right\|_{L^{2}(G)}+\left\|u_{1}\right\|_{L^{2}(G)}  )$ given in  (\ref{111111}).
	
	Here we  briefly recall   the notion of mild solutions  in our framework to  the Cauchy problem (\ref{eq0010})  and  will analyze our approach to prove Theorem \ref{thm22}. Applying Duhamel’s principle, the solution to the nonlinear inhomogeneous problem
	\begin{align}\label{eq3111}
		\begin{cases}
			\partial^2_tu+(-\mathcal{L})^\alpha u =F(t, x), & x\in G,t>0,\\
			u(0,x)=  u_0(x),  & x\in G,\\ \partial_tu(0, x)=  u_1(x), & x\in G,
		\end{cases}
	\end{align}
	can be expressed as
	$$ u(t, x)= u_{0}(x)*_{(x)}E_{0}(t, x)+u_{1}(x)*_{(x)}E_{1}(t, x) +\int_{0}^{t} F(s, x)*_{(x)} E_{1}(t-s, x) \;d s,  $$
	where $*_{(x)}$ denotes the convolution with respect to the $x$ variable, $E_{0}(t, x)$ and $E_{1}(t, x)$  are  the fundamental solutions to the homogeneous problem (\ref{eq3111}), i.e., when  $F=0$ with initial data $\left(u_{0}, u_{1}\right)=\left(\delta_{0}, 0\right)$ and $\left(u_{0}, u_{1}\right)=$ $\left(0, \delta_{0}\right)$, respectively. 
	For any left-invariant differential operator $L$ on the compact Lie group $ {G}$, we applied  the property  that $L\left(v*_{(x)} E_{1}(t, \cdot)\right)=v *_{(x)} L\left(E_{1}(t, \cdot)\right)$ and   the invariance by time translations for the wave operator $	\partial^2_t+(-\mathcal{L})^\alpha $
	in order to get the previous representation formula.

	Thus, the function  $u$ is  said to be a mild solution to (\ref{eq3111})  on $[0, T]$ if $u$ is a fixed point for  the       integral operator  $N: u \in X(T) \rightarrow N u(t, x) $ defined as 
	\begin{align}\label{f2} 
		N u(t, x)= \varepsilon u_{0}(x) *_{(x)}  E_{0}(t, x)+\varepsilon u_{1}(x) *_{(x)}  E_{1}(t, x) +\int_{0}^{t}|u(s, x)|^{p} *_{(x)}  E_{1}(t-s, x) \;ds
	\end{align}
	in the evolution space $X(T) \doteq \mathcal{C}\left([0, T], H_{\mathcal{L}}^{\alpha}(G)\right) \cap \mathcal{C}^{1}\left([0, T], L^{2}(G)\right)$, equipped with the norm defined in  (\ref{eq33333}). 
	In order to prove $N$ admits a uniquely determined fixed point for sufficiently small $T=T(\varepsilon)$, we use Banach's fixed point theorem with respect to the norm on $X(T)$ as defined above.   More preciously, for $\left\|\left(u_{0}, u_{1}\right)\right\|_{H_{\mathcal{L}}^{\alpha}(G) \times L^{2}(G)}$   small enough,  if we can show the validity of the following two  inequalities
	$$\|N u\|_{X(T)} \leq C\left\|\left(u_{0}, u_{1}\right)\right\|_{H_{\mathcal{L}}^{\alpha}(G) \times L^{2}(G)}+C\|u\|_{X(T)}^{p},$$
	$$\|N u-N v\|_{X(T)} \leq C\|u-v\|_{X(T)}\left(\|u\|_{X(T)}^{p-1}+\|v\|_{X(T)}^{p-1}\right),$$
	for any $u, v \in X(T)$ and for some  suitable constant $C>0$ independent of $T$. Then by Banach's fixed point theorem  we can assure  that the operator $N$ admits a uniquely determined fixed point $u$.  This     function $u$ will be  our mild solution to (\ref{eq3111})  on $[0, T]$.

	In order to prove the local existence result, an important tool is the following Gagliardo-Nirenberg type inequality. We refer to   \cite{Gall} for the detailed proof of the inequality. 
	\begin{lem}\cite{Gall} \label{lemma1}
		Let $G$ be a connected unimodular Lie group with topological dimension $n.$ For any $1<q_0<\infty,~0<q,q_1<\infty$ and $0<\alpha<n$ such that $q_0<\frac{n}{\alpha},$ the following Gagliardo-Nirenberg type inequality holds
		\begin{align}\label{eq33}
			\|f\|_{L^q(G)}\lesssim \|f\|^\theta_{H^{\alpha,q_0}_\mathcal L(G)}\|f\|^{1-\theta}_{L^{q_1}(G)}
		\end{align} 
		for all $f\in H^{\alpha,q_0}_\mathcal L(G)\cap L^{q_1}(G),$ provided that
		\begin{align*}
			\theta=\theta(n,\alpha,q,q_0,q_1)=\frac{\frac{1}{q_1}-\frac{1}{q}}{\frac{1}{q_1}-\frac{1}{q_0}+\frac{\alpha}{n}}\in[0,1].
		\end{align*}
	\end{lem} 
	We refer to \cite{Gall, 27}  for several immediate important  remarks from Lemma \ref{lemma1}.
		%
		An version of Lemma \ref{lemma1} which is useful in our setting is the following result. 
		\begin{cor}
			Let $G$ be a connected unimodular Lie group with topological dimension $n\geq 2[\alpha]+2$.  For any $q\ge2$ such that $q\leq\frac{2n}{n-2\alpha}$, the following Gagliardo-Nirenberg type inequality holds
			\begin{align}\label{eq34}
				\|f\|_{L^q(G)}\lesssim \|f\|^{\theta(n, q, \alpha)}_{H^{\alpha }_\mathcal L(G)}\|f\|^{1-\theta(n, q, \alpha)}_{L^{2}(G)}
			\end{align} 
			for all $f\in H^{\alpha}_\mathcal L(G)$, where $\theta(n, q, \alpha)=\frac{n}{\alpha}\left(\frac{1}{2}-\frac{1}{q} \right) $.
		\end{cor}
		
		\begin{proof}[Proof of Theorem \ref{thm22}]
			The  expression    (\ref{f2})  can be wriiten as  $N u=u^\sharp+I[u]$, where 
			\begin{align*}
				u^\sharp(t,x)=\varepsilon u_{0}(x) *_{(x)}  E_{0}(t, x)+\varepsilon u_{1}(x) *_{(x)}  E_{1}(t, x)
			\end{align*}
			and 
			\begin{align*}
				I[u](t,x):=\int\limits_0^t |u(s,x)|^p*_x E_1(t-s, x)ds.
			\end{align*} 
			
			Now, for the part $u^\sharp$,   Theorem \ref{thm11},   immediately  implies that
			\begin{align}\label{f3}
				\|u^\sharp\|_{X(T)}\lesssim\varepsilon\|(u_0,u_1)\|_{{H}_{\mathcal L}^\alpha (G)\times L^2(G)}.
			\end{align}
			On the other hand, for the part $I[u]$, using Minkowski's integral inequality, Young's convolution inequality, Theorem \ref{thm11}, and  by time translation invariance property of the Cauchy problem (\ref{eq0010}), we get  
			\begin{align}\label{f}\nonumber
				\|\partial_t^j(-\mathcal L)^{i\alpha/2}I[u]\|_{L^2(G)}&=\left(\int_{G}	\big |\partial_t^j(-\mathcal L)^{i\alpha/2} \int\limits_0^t |u(s,x)|^p*_x E_1(t-s, x)ds\big |^2 dg\right)^{\frac{1}{2}}\\\nonumber
				&=\left(\int_{G}\big	|  \int\limits_0^t |u(s,x)|^p*_x \partial_t^j(-\mathcal L)^{i\alpha/2}E_1(t-s, x)ds\big|^2 dg\right)^{\frac{1}{2}}	\\\nonumber
				&\lesssim  \int\limits_0^t \| |u(s,\cdot )|^p*_x \partial_t^j(-\mathcal L)^{i\alpha/2}E_1(t-s, \cdot)\|_{L^2(G)}ds\\\nonumber
				&\lesssim  \int\limits_0^t \| u(s,\cdot)^p\|_{L^2(G)} \|\partial_t^j(-\mathcal L)^{i\alpha/2}E_1(t-s, \cdot)\|_{L^2(G)}ds\\\nonumber
				&\lesssim \int\limits_0^t (t-s)^{1-(j+i)}\|u(s,\cdot)\|^p_{L^{2p}(G)}ds\\\nonumber
				&\lesssim\int\limits_0^t (t-s)^{1-(j+i)} \|u(s,\cdot)\|^{p\theta(n,2p, \alpha)}_{H^\alpha_\mathcal L(G)}\|u(s,\cdot)\|^{p(1-\theta(n,2p,\alpha ))}_{L^2(G)}ds \\\nonumber
				&\lesssim\int\limits_0^t (t-s)^{1-(j+i)} a(s)^p  \|u\|^p_{X(s)}ds \\ 	 	&\lesssim t^{2-(j+i)} a(t)^p   \|u\|^p_{X(t)},
			\end{align} 
			for $i,j\in\{0,1\}$ such that $0\leq i+j\leq 1.$   Again for $i,j\in\{0,1\}$ such that $0\leq i+j\leq 1,$ a similar calculations as in  (\ref{f}) together with H\"older's inequality  and    (\ref{eq34}),  we get 
			\begin{align}\label{f5}\nonumber
				& 	\|\partial_t^j(-\mathcal L)^{i\alpha/2}\left(I[u]-I[v]\right)\|_{L^2(G)}\\\nonumber&\lesssim \int\limits_0^t (t-s)^{1-(j+i)}\||u(s,\cdot)|^p-|v(s,\cdot)|^p\|_{L^{2}(G)}ds\\\nonumber
				&\lesssim\int\limits_0^t (t-s)^{1-(j+i)} \|u(s,\cdot)-v(s,\cdot)\|_{L^{2p}(G)}\left(\|u(s,\cdot)\|^{p-1}_{L^{2p}(G)}+\|v(s,\cdot)\|^{p-1}_{L^{2p}(G)}\right)ds\\ &\lesssim t^{2-(j+i)} a(t)^p \|u-v\|_{X(t)}\left(\|u\|^{p-1}_{X(t)}-\|v\|^{p-1}_{X(t)}\right).
			\end{align} 
			Thus 	combining  (\ref{f3}),   (\ref{f}), and (\ref{f5}), we have
			\begin{align}\label{1}
				\|N u\|_{X(t)} \leq D \varepsilon\left\|\left(u_{0}, u_{1}\right)\right\|_{H_{\mathcal{L}}^{\alpha }(G) \times L^{2}(G)}+D(1+t)^{b(p)}\|u\|_{X(t)}^{p} 
			\end{align} 
			and 
			\begin{align}\label{2} \|Nu-Nv\|_{X(T)}\leq   D(1+t)^{b(p)}\|u-v\|_{X(t)}\left(\|u\|^{p-1}_{X(T)}-\|v\|^{p-1}_{X(T)}\right),\end{align} 
			where
			$$
			b(p) =\left\{\begin{array}{ll}
				p+1 & \text { if } u_{1} \neq 0 ,\\
				2 & \text { if } u_{1}=0.
			\end{array}\right.
			$$ 
			Let  us consider $L_{0}=\left\|u_{0}\right\|_{H_{\mathcal{L}}^{\alpha }(G)}+\|u_1\|_{L^2(G)}$. Then       for any $L \geq 2 D L_{0}$ and any $t \leq C \varepsilon^{-\frac{p-1}{b(p)}}$ with $C \doteq(4 D L)^{-\frac{1}{b(p)}}$, from (\ref{1}) and (\ref{2}),  we get 
			$$
			\|N u\|_{X(t)} \leq  L \varepsilon, \quad\|N u-N v\|_{X(t)} \leq \frac{1}{2}\|u-v\|_{X(t)},
			$$
			for all $u, v \in {B}(0, L \varepsilon) \doteq\left\{u \in X(t):\|u\|_{X(t)} \leq  L \varepsilon\right\}$.  
			This shows that the map $N$ turns out to be a contraction in  the ball $ {B}(0, L \varepsilon)$ in the Banach space $X(T).$   Therefore, the Banach's fixed point theorem gives us the uniquely determined fixed point $u$ for the map $N$ which is   our mild solution to the system (\ref{eq0010}) on $[0, t] \subset[0, T(\varepsilon)]$.    Moreover,  from the above discussions,  we  have  the following lower bound estimates $$T(\varepsilon) \geq C \varepsilon^{-\frac{p-1}{b(p)}}.$$ This   completes  the proof of Theorem \ref{thm22}.
		\end{proof}

		From the above local existence result, we have the following remarks.
		\begin{rmk}
			We note that in the statement of Theorem \ref{thm22}, the  restriction on the upper bound for the exponent $p$ which is $p\leq\frac{n}{n-2\alpha }$   is necessary in order to apply Gagliardo-Nirenberg type inequality (\ref{eq34}) in  (\ref{f}). Also   the other restriction $n \geq 2[\alpha]+2$ is made to fulfill the assumptions for the employment of such inequality.
		\end{rmk}		
		
		
		\begin{rmk}
			For the   Euclidean or in the Heisenberg group, an additional $L^{1 }$-regularity for the Cauchy data is required to get a global existence result for a non-empty range for $p$.   Using this trick, one can improve the $L^{2}$-decay estimates for the solution and its first-order derivatives to the linear homogeneous problem.   Since $L^{2}(G) \subset L^{1}(G)$ in the case of  compact group,   no additional decay rate can be gained for the $L^{2}(G)$-norm of the solution  even working with $L^{1}(G)$-regularity for $u_{0}, u_{1}$.
			
		\end{rmk}

		\section{Blow-up result}\label{sec5}
		In this section,  we      prove Theorem \ref{f6}   using a comparison argument for ordinary differential inequality of   second order. First we recall   the following improved Kato’s lemma
		for  ordinary differential inequality  with upper bound estimate for the lifespan.
		\begin{lem}\cite{kato}\label{kato}
			Let $p > 1, a > 0, q > 0$ satisfying
			$$M :=\frac{p-1}{2}a-\frac{q}{2}+1>0
			.$$
			Assume that $F \in C^2 	([0, T))$  satisfy
			\begin{align}
				F(t) \geq A t^{a} & \text { for } t \geq T_{0}, \\
				F^{\prime \prime}(t) \geq B(t+R)^{-q}|F(t)|^{p} & \text { for } t \geq 0, \\\label{kato1}
				F(0) \geq 0, F^{\prime}(0)>0, &
			\end{align}
			where $A, B, R, T_{0}$ are positive constants. Then, there exists a positive constant $C_{0}=C_{0}(p, a, q, B)$ such that
			$$
			T<2^{\frac{2}{M}} T_{1}
			$$
			holds provided
			$$
			T_{1}:=\max \left\{T_{0}, \frac{F(0)}{F^{\prime}(0)}, R\right\} \geq C_{0} A^{-\frac{(p-1) }{2 M}}.
			$$
		\end{lem}

		\begin{lem}\label{kato2}
			Assume that (\ref{kato1}) is replaced by
			$$
			F(0)>0, F^{\prime}(0)=0
			$$
			and additionally that there is a time $t_{0}>0$ such that
			$$
			F\left(t_{0}\right) \geq 2 F(0) .
			$$
			Then, the conclusion of Lemma \ref{kato}  is changed to that there exists a positive constant $C_{0}=C_{0}(p, a, q, B)$ such that
			$$
			T<2^{\frac{2}{M}} T_{2}
			$$
			holds provided
			$$
			T_{2}:=\max \left\{T_{0}, t_{0}, R\right\} \geq C_{0} A^{-\frac{(p-1) }{2 M}}.
			$$
		\end{lem}
		Now we are ready to prove our main result of this section using Lemma \ref{kato} and Lemma \ref{kato2}. 
		\begin{proof}[Proof of Theorem \ref{f6}]
			According to Definition \ref{eq332},  let $u$ be a local in-time energy solution to (\ref{eq0010})  with lifespan $T$. Let    $t \in(0, T)$ be fixed.  Suppose that $\phi \in \mathcal{C}_{0}^{\infty}([0, T) \times G),$  is a cut-off function such that $\phi=1$ on $[0, t] \times G$ in (\ref{eq011}). Then 
			\begin{align}\label{f7}
				\int_{G} \partial_{t} u(t, x) \;dx-\varepsilon \int_{G} u_{1}(x) \;dx=\int_{0}^{t} \int_{\mathrm{G}}|u(s, x)|^{p}  {~d} x {~d} s
			\end{align}
			Let us introduce the time-dependent functional
			$$
			U_{0}(t) \doteq \int_{G} u(t, x) \;dx. 
			$$
			Then the equality (\ref{f7})    can be rewritten in the following way:
			$$
			U_{0}^{\prime}(t)-U_{0}^{\prime}(0)=\int_{0}^{t} \int_{G}|u(s, x)|^{p} \;dx \;ds .
			$$
			We also remark that,  from the assumptions on the initial data, we obtain 
			$$
			U_{0}(0)=\varepsilon \int_{G} u_{0}(x) \;dx \geq 0 \quad \text { and } \quad U_{0}^{\prime}(0)=\varepsilon \int_{G} u_{1}(x) \;dx \geq 0.
			$$
			
			Then, $U_{0}^{\prime}$ is differentiable with respect to $t$, and by  Jensen's inequality, we have 
			$$
			U_{0}^{\prime \prime}(t)=\int_{\mathrm{G}}|u(t, x)|^{p} \;dx \geq\left|U_{0}(t)\right|^{p}.
			$$
			Now applying  Lemmas \ref{kato} and \ref{kato2} to the functional $U_{0}$ we conclude the proof of Theorem \ref{f6}.

		\end{proof}

			%
			%
		%
		%
		
		\section{Fractional Klein-Gordon equation with regular mass}\label{sec6}
		In this section, we consider the fractional  Klein-Gordon equation on the  compact Lie group $G$  with regular mass term depending on the space variable. More preciously,   for $T>0$ and for $0<\alpha<1$, we consider the   Cauchy problem (\ref{eq1111111}), namely, 
		\begin{align*} 
			\begin{cases}
				\partial^2_tu(t, x)+\left( -\mathcal{L}\right)^\alpha u(t, x)+m(x)u(t, x)=0, & (x, t)\in[0, T]\times G,\\
				u(0,x)=u_0(x),  & x\in G,\\ \partial_tu(x,0)=u_1(x), & x\in G,
			\end{cases}
		\end{align*}
		where $u_{0}(x)$ and $u_{1}(x)$ are two given functions on $G$.	Here, the mass function $m$ is supposed to be non-negative and  regular function. Note that when the mass function $m$ is identically zero, then the above system becomes usual fractional wave equation       defined in (\ref{eq0010}).  Now we prove main result of this section.  
		
		\begin{proof}[Proof of Theorem \ref{thm11111}]
			By	multiplying equation (\ref{eq1111111}) by $u_{t}$ and integrating over $G$ to get
			\begin{align}\label{eq222}
				\nonumber	\operatorname{Re}\Big(\left\langle u_{t t}(t, \cdot), u_{t}(t, \cdot)\right\rangle_{L^{2}(G)}&+\left\langle\left( -\mathcal{L}\right)^\alpha u(t, \cdot), u_{t}(t, \cdot)\right\rangle_{L^{2}(G)}\\&\quad\qquad\qquad+\left\langle m(\cdot) u(t, \cdot), u_{t}(t, \cdot)\right\rangle_{L^{2}(G)}\Big)=0, 
			\end{align}
			for all $t \in[0, T]$. 
			
			Now it is easy to check that
			$$
			\begin{array}{l}
				\operatorname{Re}\left(\left\langle u_{t t}(t, \cdot), u_{t}(t, \cdot)\right\rangle_{L^{2}(G)}\right)=\frac{1}{2} \partial_{t}\left\langle u_{t}(t, \cdot), u_{t}(t, \cdot)\right\rangle_{L^{2}(G)} \\
				\operatorname{Re}\left(\left\langle\left( -\mathcal{L}\right)^\alpha u(t, \cdot), u_{t}(t, \cdot)\right\rangle_{L^{2}(G)}\right)=\frac{1}{2} \partial_{t}\left\langle\left( -\mathcal{L}\right)^{\frac{\alpha}{2}} u(t, \cdot), \left( -\mathcal{L}\right)^{\frac{\alpha}{2}} u(t, \cdot)\right\rangle_{L^{2}(G)},
			\end{array}
			$$
			and
			$$
			\operatorname{Re}\left(\left\langle m(\cdot) u(t, \cdot), u_{t}(t, \cdot)\right\rangle_{L^{2}(G)}\right)=\frac{1}{2} \partial_{t}\langle\sqrt{m}(\cdot) u(t, \cdot), \sqrt{m}(\cdot), u(t, \cdot)\rangle_{L^{2}(G)} .
			$$
			Let us denote  
			$$
			E(t):=\left\|u_{t}(t, \cdot)\right\|_{L^{2}(G)}^{2}+\left\|\left( -\mathcal{L}\right)^{\frac{\alpha}{2}} u(t, \cdot)\right\|_{L^{2}(G)}^{2}+\|\sqrt{m}(\cdot) u(t, \cdot)\|_{L^{2}(G)}^{2},
			$$
			as 	the energy functional of the system (\ref{eq1111111}). Then   equation (\ref{eq222}) implies that $\partial_{t} E(t)=0$, and consequently  we have  $E(t)=E(0)$, for all $t \in[0, T]$. Since $
			\left\|\left( -\mathcal{L}\right)^{\frac{\alpha}{2}} u_{0}\right\|_{L^{2}(G)},  \left\|u_{0}\right\|_{L^{2}(G)} \leq\left\|u_{0}\right\|_{H_{\mathcal{L}}^\alpha(G)},
			$ by taking into  consideration the fact that 
			$$
			\left\|\sqrt{m}(\cdot) u_{0}(\cdot )\right\|_{L^{2}(G)}^{2} \leq\|m\|_{L^{\infty}(G)}\left\|u_{0}\right\|_{L^{2}(G)}^{2},
			$$
			each positive term  of  $E(t)$  is bounded by  itself, i.e., 	\begin{align}\label{eq006}\nonumber
				\left\|u_{t}(t, \cdot)\right\|_{L^{2}(G)}^{2}, \left\|\left( -\mathcal{L}\right)^{\frac{\alpha}{2}} u(t, \cdot)\right\|_{L^{2}(G)}^{2} & \lesssim\left\|u_{1}\right\|_{L^{2}(G)}^{2}+\left\|u_{0}\right\|_{H_{\mathcal{L}}^\alpha(G)}^{2}+\|m\|_{L^{\infty}(G)}\left\|u_{0}\right\|_{L^{2}(G)}^{2} \\
				& \lesssim \left(1+\|m\|_{L^{\infty}(G)}\right)\left\{ \left\|u_{0}\right\|_{H^{\alpha}(G)}^{2}+\left\|u_{1}\right\|_{L^{2}(G)}^{2}\right\},
			\end{align}
			
			and 
			\begin{align}\label{eq003}
				\|\sqrt{m}(\cdot) u(t, \cdot)\|_{L^{2}(G)}^{2} \lesssim\left(1+\|m\|_{L^{\infty}(G)}\right)\left\{ \left\|u_{0}\right\|_{H^{\alpha}(G)}^{2}+\left\|u_{1}\right\|_{L^{2}(G)}^{2}\right\},
			\end{align}
			uniformly for  $t \in[0, T]$.  Thus the desired estimates for  $	\left\|u_{t}(t, \cdot)\right\|_{L^{2}(G)}^{2}, \left\|\left( -\mathcal{L}\right)^{\frac{\alpha}{2}} u(t, \cdot)\right\|_{L^{2}(G)}^{2} $ are  proved.
			
			In order to prove  (\ref{eq2}), it remains to  calculate the estimate for the norm  $\|u(t, \cdot)\|_{L^{2}(G)}$. Let $\widehat{u}(t, \xi)=(\widehat{u}(t, \xi)_{kl})_{1\leq k, l\leq d_\xi}\in \mathbb{C}^{d_\xi\times d_\xi}, [\xi]\in\widehat{ G}$ denote the group Fourier transform of $u$  with respect to the $x $ variable.   Invoking group Fourier transform with respect to $x$ on   (\ref{eq1111111}), we get the following  Cauchy problem for the system of ODEs (with size of the system that depends on the representation $\xi$)
			\begin{align}\label{eq666}
				\begin{cases}
					\partial^2_t\widehat{u}(t,\xi)+(-	\sigma_{\mathcal{L}}(\xi))^\alpha \widehat{u}(t,\xi) =\widehat{f}(t,\xi),& [\xi]\in\widehat{ G},~t>0,\\ \widehat{u}(0,\xi)=\widehat{u}_0(\xi), &[\xi]\in\widehat{ G}\\ \partial_t\widehat{u}(0,\xi)=\widehat{u}_1(\xi), &[\xi]\in\widehat{ G},
				\end{cases} 
			\end{align}
			where  $\sigma_{\mathcal{L}}$	is the symbol of the  of the Laplace-Beltrami operator operator $\mathcal{L}$ and 
			$\widehat{f}(t, \xi)$ denotes the group Fourier transform of the function $f(t, x)=-m(x) u(t, x)$. Using the identity (\ref{symbol}),  the   system  (\ref{eq666}) is decoupled in $d_\xi^2$ independent   ODEs, namely,
			\begin{align}\label{eq7}
				\begin{cases}
					\partial^2_t\widehat{u}(t,\xi)_{kl}+ \lambda_\xi^{2\alpha }  \widehat{u}(t,\xi)_{kl}= \widehat{f}(t,\xi)_{kl},& [\xi]\in\widehat{ G},~t>0,\\ \widehat{u}(0,\xi)_{kl}=\widehat{u}_0(\xi)_{kl}, &[\xi]\in\widehat{ G}\\ \partial_t\widehat{u}(0,\xi)_{kl}=\widehat{u}_1(\xi)_{kl}, &[\xi]\in\widehat{ G},
				\end{cases}
			\end{align}
			for all $k,l\in\{1,2,\ldots,d_\xi\}.$ 
			
			Then    the characteristic equation of (\ref{eq7}) is given by
			\[\lambda^2+\lambda_\xi^{2\alpha } =0,\]
			and consequently the characteristic roots are   $\lambda=\pm i \lambda_\xi^\alpha  $. 
			Thus the solution of the   homogeneous  equation of (\ref{eq7}) is given by 
			
				\begin{align} 
				&\widehat{u}(t,\xi)_{kl}=\begin{cases}
				\cos \left(t \lambda_\xi^\alpha \right) \widehat{u}_0(\xi)_{kl}+\frac{\sin \left(t \lambda_\xi^\alpha \right)}{\lambda_\xi^\alpha } \widehat{u}_1(\xi)_{kl}& \text{if }\lambda_\xi^\alpha\neq 0,\\
						  \widehat{u}_0(\xi)_{kl}+t \widehat{u}_1(\xi)_{kl}& \text{if }\lambda_\xi^\alpha= 0.
				\end{cases}  
			\end{align}
			For $\lambda_\xi^\alpha\neq 0$, applying Duhamel's principle,  the solution of (\ref{eq7}) is given by 
			\begin{align}\label{eq001}
				\widehat{u}(t,\xi)_{kl}=\cos \left(t \lambda_\xi^\alpha \right) \widehat{u}_0(\xi)_{kl}+\frac{\sin \left(t \lambda_\xi^\alpha \right)}{\lambda_\xi^\alpha } \widehat{u}_1(\xi)_{kl}+\int_{0}^{t} \frac{\sin \left((t-s) \lambda_\xi^\alpha\right)}{\lambda_\xi^\alpha} \widehat{f}(s,\xi)_{kl} \;{d} s .
			\end{align}
			Without loss of generality, we assume  that $T \geq 1$, and we recall  the following estimates
			$$
			\left|\cos \left(t \lambda_\xi^\alpha\right)\right| \leq 1, \quad \forall t \in[0, T],
			$$
			and
			$$
			\left|\sin \left(t \lambda_\xi^\alpha\right)\right| \leq 1,
			$$
			for large values of the quantities $t \lambda_\xi^\alpha$, while for small values of them, one have 
			$$
			\left|\sin \left(t \lambda_\xi^\alpha\right)\right| \leq t \lambda_\xi^\alpha \leq T \lambda_\xi^\alpha.
			$$
			Using the Cauchy-Schwarz inequality, the equation  (\ref{eq001}) yields
			\begin{align}\label{neqeq}\nonumber
				|	\widehat{u}(t,\xi)_{kl}|&\leq| \widehat{u}_0(\xi)_{kl}| +T |\widehat{u}_1(\xi)_{kl}|+ \int_{0}^{t} |t-s|| \widehat{f}(s,\xi)_{kl} |\;{d} s \\
				&\leq| \widehat{u}_0(\xi)_{kl}| +T |\widehat{u}_1(\xi)_{kl}|+\|t-s\|_{L^{2}[0, T]} \|\widehat{f}(\cdot ,\xi)_{kl}\|_{L^{2}[0, T]}. 	\end{align}
 For $\lambda_\xi^\alpha= 0$, proceeding  likewise as above, we get  the same  estimate (\ref{neqeq}) for $	|	\widehat{u}(t,\xi)_{kl}|$.   
			Now  from the last estimate (\ref{neqeq}),  by substituting back our initial functions in $t$, gives
			$$
			\left|\widehat{u}(t, \xi)_{k l}\right|^{2} \lesssim\left|\widehat{u}_{0}(\xi)_{k l}\right|^{2}+\left|\widehat{u}_{1}(\xi)_{kl}\right|^{2}+|| \widehat{f}(\cdot , \xi)_{k l} \|_{L^{2}[0, T]}^{2},
			$$
			where the latter holds uniformly in $\xi \in \widehat{G}$ and for each $k,l\in\{1,2,\ldots,d_\xi\}.$   Thus summing the above over $k, l$,  we get
			$$
			\sum_{k, l=1}^{d_{\xi}}	\left|\widehat{u}(t, \xi)_{k l}\right|^{2} \lesssim \sum_{k, l=1}^{d_{\xi}}	\left|\widehat{u}_{0}(\xi)_{k l}\right|^{2}+\sum_{k, l=1}^{d_{\xi}}	\left|\widehat{u}_{1}(\xi)_{kl}\right|^{2}+\sum_{k, l=1}^{d_{\xi}}	\int_{0}^{T} | \widehat{f}(t, \xi)_{k l} |^{2}\;dt.
			$$
			This implies that

			$$
			\left\|\widehat{u}(t, \xi)_{k l}\right\|_{\mathrm{HS}}^{2} \lesssim\left\|\widehat{u}_{0}(\xi)_{k l}\right\|_{\mathrm{HS}}^{2}+\left\|\widehat{u}_{1}(\xi)_{k l}\right\|_{\mathrm{HS}}^{2}+ 	\int_{0}^{T} \sum_{k, l=1}^{d_{\xi}}| \widehat{f}(t, \xi)_{k l} |^{2}\;dt
			$$
			and consequently  
			$$
			\sum_{[\xi] \in \widehat{G}} d_{\xi}	\left\|\widehat{u}(t, \xi)_{k l}\right\|_{\mathrm{HS}}^{2} \lesssim 	\sum_{[\xi] \in \widehat{G}} d_{\xi}\left\|\widehat{u}_{0}(\xi)_{k l}\right\|_{\mathrm{HS}}^{2}+	\sum_{[\xi] \in \widehat{G}} d_{\xi}\left\|\widehat{u}_{1}(\xi)_{k l}\right\|_{\mathrm{HS}}^{2}+ 		\sum_{[\xi] \in \widehat{G}} d_{\xi}\int_{0}^{T}  \|\widehat{f}(t, \xi)\|_{\operatorname{HS}}^{2}\;dt.
			$$
			Using    dominated convergence theorem, Fubini's theorem, and  the Plancherel identity (\ref{eq002}), we get 
			\begin{align}\label{eq005}
				\|u(t, \cdot)\|_{L^{2}(G)}^{2} \lesssim\left\|u_{0}\right\|_{L^{2}(G)}^{2}+\left\|u_{1}\right\|_{L^{2}(G)}^{2}+\int_{0}^{T}\|f(t, \cdot)\|_{L^{2}(G)}^{2}\;dt .
			\end{align}
			Now by (\ref{eq003}), we have
			$$
			\begin{aligned}
				\|f(t, \cdot)\|_{L^{2}(G)}^{2} &=\|m(\cdot) u(t, \cdot)\|_{L^{2}(G)}^{2} \\
				& \leq\|m\|_{L^{\infty}(G)}\|\sqrt{m}(\cdot) u(t, \cdot)\|_{L^{2}(G)}^{2} \\
				& \lesssim\left(1+\|m\|_{L^{\infty}(G)}\right)^{2}\left\{\left\|u_{0}\right\|_{H^{\alpha}(G)}^{2}+\left\|u_{1}\right\|_{L^{2}(G)}^{2}\right\} .
			\end{aligned}
			$$
			Thus from (\ref{eq005}), we get 
			\begin{align}\label{eq004}
				\|u(t, \cdot)\|_{L^{2}(G)}^{2} \lesssim\left(1+\|m\|_{L^{\infty}(G)}\right)^{2}\left\{\left\|u_{1}\right\|_{L^{2}(G)}^{2}+\left\|u_{0}\right\|_{H^{\alpha}(G)}^{2}\right\},
			\end{align}
			uniformly in $t \in[0, T]$. Hence the required inequality    (\ref{eq2})   follows from   (\ref{eq006}) and  (\ref{eq004}). Moreover, the uniqueness of $u$ is an immediate consequence of the inequality (\ref{eq2}) and this completes the  proof of the  theorem.
		\end{proof}

		\section{Final remarks}\label{sec7}
		From  Remark (\ref{remark})  one can see that  the sharp lifespan estimates  	for local in-time solutions to (\ref{eq0010}) is    indepenedent of $\alpha$.  Thus,  for any $0<\alpha<1$, the lifespan estimates  for solutions  to the Cauchy problem for the fractional wave equation  (\ref{eq0010})  will be same as the  sharp lifespan estimates for 
		the  semilinear wave equation  on compact Lie group $G$ proved in  \cite{27}.

		In this  work we         see that under some suitable assumptions on the initial data,     local in-time solutions to  the Cauchy problem for   fractional wave equation blow up in finite time for any $p > 1$.  In other words, we don't have any global existence result.  However,    we believe that    the presence of a positive damping term and of a positive mass term in the Cauchy problem completely reverses the scenario.   
		This is possible because  without requiring any additional lower bound for $p > 1$, 		  the global in time existence of small data solution can be proved in the evolution energy space. This will be pursued in a forthcoming paper.

		\section{Data availability statement}
		The authors confirm that the data supporting the findings of this study are available within the article  and its supplementary materials.

	\end{document}